\documentclass{aptpub}
\authornames{Cornean, Herbst, M\o ller, St\o ttrup and S\o rensen} 
\shorttitle{Random variables with stationary digits} 

\usepackage{mathtools}

\usepackage{graphicx}
\usepackage{color}
\usepackage{centernot}
\newcommand{\R}{\mathbb{R}}

\newcommand{\dd}{\mathrm{d}}
\newcommand{\rmP}{\mathrm{P}}

\numberwithin{equation}{section}  

\begin{document}

\title{Characterization of random variables\\ with stationary digits} 

\authorone[Aalborg University]{Horia D. Cornean} 
\authortwo[University of Virginia]{Ira W. Herbst}
\authorone[Aalborg University]{Jesper M{\O}ller}
\authorone[Aalborg University]{Benjamin B. St{\O}ttrup}
\authorone[Aalborg University]{Kasper S. S{\O}rensen}

\addressone{Department of Mathematical Sciences, Aalborg University, Skjernvej 4A, 9220 Aalborg, Denmark} 
\addresstwo{Department of Mathematics, University of Virginia, Charlottesville, VA 22903, USA}

\begin{abstract}
    		Let $q\ge2$ be an integer, $\{X_n\}_{n\geq 1}$ a stochastic process with state space $\{0,\ldots,q-1\}$, and $F$ the cumulative distribution function (CDF) of 
    		$\sum_{n=1}^\infty X_n q^{-n}$. We show that stationarity of $\{X_n\}_{n\geq 1}$ is equivalent to a functional equation obeyed by $F$ and use this to characterize the characteristic function of $X$ and 
    		the structure of $F$ in terms of its Lebesgue decomposition. More precisely, while the absolutely continuous component of $F$ can only be the uniform distribution on the unit interval, its discrete component can only be a countable convex combination of certain explicitly computable CDFs for probability distributions with finite support. 
    		We also show that $\dd F$ is a Rajchman measure if and only if $F $ is the uniform CDF on $[0,1]$. 
\end{abstract}


         
\section{Introduction}\label{s:intro}

Consider a random variable $X$ on the unit interval $[0,1]$ which is given by the base-$q$ expansion
\begin{equation}\label{e:X}
X:=(0.X_1X_2\ldots)_q:=\sum_{n=1}^\infty X_n q^{-n},
\end{equation}
where $q\in\mathbb N$ (the set of natural numbers), and where the digits  
$\{X_n\}_{n\ge1}$ form a stochastic process with values in $\{0,\ldots,q-1\}$. The case where the $X_n$'s are independent identically distributed (IID) has been much studied in the literature (see Sections~\ref{s:background} and \ref{s:future}). The present paper deals with the more general case where $\{X_n\}_{n\geq 1}$ is \emph{stationary}, i.e., when $\{X_n\}_{n\ge1}$ and $\{X_n\}_{n\ge2}$ are identically distributed -- for short we refer to this setting as \emph{stationarity}.
As we will see, under stationarity there is almost surely a one-to-one correspondence between $\{X_n\}_{n\ge1}$ and $X$, and we will 
study various properties of  the cumulative distribution function (CDF) of $X$ given by
\begin{equation}\label{e:Fdef}
F(x):=\mathrm P(X\le x),\qquad x\in\mathbb R
\end{equation} 
and its associated probability measure $\mathrm d F$.

\subsection{Background}\label{s:background}

Assuming the $X_n$'s are IID, the stochastic process $\{X_n\}_{n\geq 1}$ is a so-called Bernoulli scheme. Then, in the dyadic case $q=2$, ignoring the trivial case with $\mathrm P(X_1=0)=1$ or $\mathrm P(X_1=1)=1$, only two different things can happen: if the digits $0$ and $1$ are equally likely, then $F$ is the uniform CDF (on $[0,1]$); otherwise, $F$ is singular  (i.e., $F$ is non-constant and differentiable almost everywhere with $F'(x)=0$), continuous, and strictly increasing on $[0,1]$, cf.\ \cite{Takacs1978,Riesz1955,Salem1943}. 
In the triadic case $q=3$, if $0$ and $2$ are equally likely and $\mathrm P(X_1 = 1) = 0$, then $F$ is the Cantor function, cf.\ Problem 31.2 in \cite{Billingsley1995}. This function is also singular continuous, but piecewise constant and only increasing on the Cantor set. In fact, interestingly, the measures $\mathrm d F$ in all the Bernoulli schemes for any $q$ are again all singular with respect to one another (this seems to be a folk theorem but see Section 14 in \cite{Billingsley1965}) and only one is absolutely continuous relative to Lebesgue measure and that is the one where all $j \in \{0,\dots, q-1\}$ are equally likely.  In the latter case, $\mathrm d F$ is Lebesgue measure itself on $[0,1]$.

Harris in \cite{Harris1955} considered the case where $q\geq 2$ and  $\{X_n\}_{n\geq 1}$ is stationary and of a mixing type. He showed that either $F$ is the uniform CDF, or $F$ has a single jump of magnitude $1$ at one of the points ${k}/{(q-1)}$, $k=0,\dots,q-1$, or $F$ is singular continuous. A similar result has been shown in \cite{Dym1968} under the assumption that $\{X_n\}_{n\geq 1}$ is stationary and ergodic, namely that either $F$ is the uniform CDF, or $F$ has $k$ jumps of magnitude $k^{-1}$, or $F$ is singular continuous.

It is well-known that if the $X_n$'s are IID, then $F$ is the uniform CDF if and only if $\dd F$ is a Rajchman measure \cite{Salem1943,Reese1989,Hu2001}. Recall that a Rajchman measure is a finite measure whose characteristic function $\mathrm{E} \mathrm{e}^{itX}$  tends to zero as $t \to\pm\infty$. Rajchman measures have received much attention in the Fourier analysis community (see the review article \cite{RLyons}) and the behavior of the characteristic function of singular continuous probability measures at infinity are of general interest in quantum mechanics (see \cite{BCM1997} and references therein). To the best of our knowledge it has yet not been clarified in the literature if 
under stationarity and when $F$ is not the uniform CDF on $[0,1]$ there is a case where
$\mathrm d F$ is a Rajchman measure.

\subsection{Our results}

In this paper, Theorem~\ref{HC1} provides a complete characterization of stationarity in terms of a functional equation for $F$ without using the extra assumptions of \cite{Harris1955} and \cite{Dym1968}. This leads to Theorem \ref{fourier} which characterizes stationarity 
in terms of the characteristic function of $X$, and the asymptotic behavior of the characteristic function at $\pm \infty$ is treated in the stationary case. 
In particular, we show that none of the measures $\dd F$ arising are Rajchman measures, except when $\dd F$ is Lebesgue measure on $[0,1]$. Furthermore, Theorem~\ref{thm1} describes stationarity in terms of a Lebesgue decomposition result for $F$. Here, it is known that the absolutely continuous component of $F$ can only be the uniform CDF (see Theorems~1 and 2 in \cite{parry}), but we give a simpler proof (see Proposition~\ref{prop19-1}). In addition, we prove that the atomic component of $F$ can only be a countable convex combination of certain explicitly computable CDFs for probability distributions with finite support.

For ease of presentation, the proofs of our theorems and propositions are deferred to Section~\ref{s:proofs}. Moreover, Section~\ref{s:proofs} provides an interesting example of a function not belonging to $L^1$ (Example~\ref{ex:7}) and another interesting example of a non-measurable function (Example~\ref{ex:8}) both of which satisfy an important requirement (but not all requirements)
of a putative probability density function for $X$ in the stationary case.

\subsection{Future work}\label{s:future}
From Theorem~\ref{thm1} it is reasonable to expect that many well-known stationary stochastic processes with a finite state space correspond to singular continuous $F$. Assuming that the $X_n$'s only take values $0$ and $1$, a natural generalization of \eqref{e:X} would be to consider 
\begin{equation*}
    X=\sum_{n=1}^\infty X_n \lambda^n,
\end{equation*}
where $\lambda\in (0,1)$. 
Let $\nu_\lambda$ denote the probability distribution of the affine transformation 
\[\sum_{n=1}^\infty (2X_n-1) \lambda^{n-1}=(2/\lambda)X-1/(1-\lambda).\] 
If the $X_n$'s are IID, $\nu_{\lambda}$ is 
a Bernoulli convolution. This is a much studied case in the literature, and in particular the absolute continuity or singularity and the Hausdorff dimension of the support of $\nu_\lambda$ as a function of $\lambda$ have been of interest (see \cite{Yuval,Varju} and references therein). We leave it as an open problem to study the absolute continuity or singularity of $\nu_\lambda$ in the more general case where $\{X_n\}_{n\geq 1}$ is stationary and $1/\lambda$ is not an integer (the present paper covers only the case where $q=1/\lambda$ is an integer). Also, in the stationary case and for any $\lambda\in (0,1)$, it would be interesting to study the Hausdorff dimension of the support of $\nu_\lambda$.


{ Define the function $f(x_1,x_2,\ldots)\coloneqq\sum_{n\geq 1} x_n\lambda^n=x$ with $x_1,x_2,\ldots\in \{0,1\}$, and for $m=1,2,\ldots$, consider the $2^m$ closed intervals of length $\lambda^{m+1}/(1-\lambda)$ and having left end points $\sum_{n=1}^m x_n\lambda^n$ with $x_1,\ldots,x_m\in \{0,1\}$. These $2^m$ intervals cover the range of $f$, which contains the state space of $X$. For $0<\lambda<1/2$, the $2^m$ intervals are disjoint, so it follows that $f$ is injective and the range of $f$ has  Lebesgue measure zero, since $2^m\lambda^{m+1}/(1-\lambda)\rightarrow0$ as $m\rightarrow\infty$. Consequently, if $0<\lambda<1/2$, the CDF of $X$ is purely singular (no matter if $\{X_n\}_{n\ge1}$ is stationary or not). Thus the case with $1/2 < \lambda < 1$ is more interesting and difficult, but a good starting point could be to study the stationary case.}


In a follow up paper we will consider the categorization of Markov chain models, renewal processes, and mixtures of these in terms of the Lebesgue decomposition of the corresponding $F$. Furthermore, in some examples of that paper we will derive closed form expressions for $F$.

\section{Main results}\label{s:main results} 

\subsection{Characterization of stationarity by a functional equation for $F$}\label{s:2.1}
Recall that
any number $x\in[0,1]$ has a base-$q$ expansion $x=(0.x_1x_2\ldots)_q$ with $x_1,x_2,\ldots\in\{0,\ldots,q-1\}$. This expansion is unique except when  
$x$ is a \emph{base-$q$ fraction} in $(0,1)$, that is, when for some (necessarily unique) $n\in\mathbb N$ we have either $x_n<x_{n+1}=x_{n+2}=\ldots=q-1$ or 
$x_n>x_{n+1}=x_{n+2}=\ldots=0$; we refer to $n$ as the \emph{order of $x$} { and denote the set of all base-$q$ fractions in $(0,1)$ by $\mathbb{Q}_q$}.

Here is the first main result of our paper. 

\begin{theorem}\label{HC1} We have the following. 
	\begin{enumerate} 
		\item[{\rm (I)}] Stationarity of $\{X_n\}_{n\geq 1}$ holds if and only if for all { $x \in \mathbb{Q}_q$} we have
		\begin{equation}\label{e:3}
		F(x) = F(0)+\sum_{j=0}^{q-1} [F((x+j)/q)-F(j/q)].
		\end{equation}
		\item[{\rm (II)}] 
		Suppose that $\tilde F$ is a CDF for a probability distribution on $[0,1]$ such that $\tilde F$ satisfies \eqref{e:3} for all { $x \in \mathbb{Q}_q$}. 
		Then there exists a unique stationary stochastic process $\{\tilde X_n\}_{n\geq 1}$ on $\{0,\ldots,q-1\}$ so that $(0.\tilde X_1\tilde X_2\ldots)_q$ follows $\tilde F$. Furthermore, $\tilde F$ is continuous at all { $x \in \mathbb{Q}_q$}, and $\tilde F$ satisfies the functional equation in \eqref{e:3} for all $x\in[0,1]$ (not just for { $x \in \mathbb{Q}_q$}).
	\end{enumerate}
\end{theorem}

\begin{remark} 
	Functional equations for the characterization of singular functions have been used in various non-probabilistic contexts, cf.\ \cite{Kairies1997}. Our 
	stationarity equation \eqref{e:3} is equivalent to special cases noticed in \cite{Kairies1997}, namely in connection to the de Rham-Tak{\'a}cs' (see the last sentence in Section 5C in \cite{Kairies1997}) and the Cantor function (see the last sentence in Section 5A in
	\cite{Kairies1997}), however, it was not noticed in \cite{Kairies1997} that \eqref{e:3} provides a characterization of stationarity as we show in Theorem~\ref{HC1}.
\end{remark}
\begin{remark}
	Clearly, when the $X_n$ are IID, (2.1) is satisfied, and for the examples of IID $X_n$ as discussed in Section 1.1, $F$ was either the uniform CDF on $[0,1]$ or a singular continuous function. 
	
	Apart from these examples, the best known example of a singular continuous CDF is probably Minkowski's question-mark function (Fragefunktion ?$(x)$) restricted to $[0,1]$, see e.g., \cite{Minkowski1904,Denjoy1932,Denjoy1934,Kairies1997}. Recall that for two reduced fractions $p/q>r/s$ such that $ ps-rq =1$ (i.e., two consecutive Farey fractions), Minkowski's question mark function is recursively defined by 
	\begin{equation*}
	    ?(0/1)=0,\quad ?(1/1)=1, \quad ?\Big(\frac{p+r}{q+s}\Big)=( ?(p/q)+?(r/s))/2,
	\end{equation*}
	and extended to any $x\in [0,1]$ by continuity. As later shown in Corollary~\ref{cor:questionmark}, the ?-function does not satisfy \eqref{e:3} for any $q\geq 2$. Hence, if the ?-function is studied in the framework of \eqref{e:X} and \eqref{e:Fdef}, the process $\{X_n\}_{n\geq 1}$ would not be stationary.
\end{remark}


In the case of a Bernoulli scheme it is well-known that $\mathrm{d}F$ is self-similar in the sense of Hutchinson, see \cite{Hutchinson1981,Strichartz1990,Hu2001}. For completeness we remind the reader that $\mathrm{d}F$ is self-similar if there exist $p_0,\dots, p_n>0$ with $\sum_{j=0}^n p_j=1$ and contractions $S_0,\dots,S_n$ on $\mathbb{R}$ such that
\[ \mathrm{d}F(E)=\sum_{j=0}^n p_j \mathrm{d} F(S_j^{-1}(E)),\]
for all Borel sets $E\subset \mathbb{R}$. 
 When $n=q-1$ and $S_j(x)=(x+j)/q$ for $j=0,\dots,q-1$, we show in the next proposition that self-similarity occurs if and only if the $X_n$'s are IID. 
 \begin{proposition}\label{prop:indep}
 Under stationarity
the $X_n$'s are IID if and only if 
\begin{equation}\label{eq:indep}
     F(x)=\sum_{j=0}^{q-1} \rmP(X_1=j) F(qx-j), \qquad x\in [0,1].
\end{equation}
\end{proposition}

As any CDF is differentiable almost everywhere, we next consider the derivative of $F$ when $F$ satisfies \eqref{e:3}. As usual, we let $L^1([0,1])$ be the set of complex absolutely integrable Borel functions defined on $[0,1]$. Note that $F'(x)$ exists outside a set of Lebesgue measure zero $M\subset [0,1]$, and for all $x$ not belonging to 
\[M\cup \{ x\in [0,1]:\; x\in qM-j \textup{ for some } j\in \{0,\dots, q-1\} \}.\]%
it follows { from \eqref{e:3}} that 
\begin{equation*}
    F'(x)=q^{-1} \sum_{j=0}^{q-1} F'((x+j)/q),
\end{equation*}
for almost all $x\in [0,1]$. 
The following Proposition~\ref{prop19-1} shows that $F'$ is almost everywhere on $[0,1]$ equal to a constant $c\in[0,1]$ (with $c=1$ if and only if $F$ is absolutely continuous), and hence that
the only purely absolutely continuous $\mathrm dF$ is Lebesgue measure. Proposition~\ref{prop19-1} is a consequence of Theorems 1 and 2 in \cite{parry} where the author considers absolutely continuous measures which are invariant under the transformation $T(x)=\beta x
\pmod 1$ where $\beta>1$. In Section~\ref{s:Proof of prop19-1} we give a proof for Proposition~\ref{prop19-1} which is simpler than the one in \cite{parry} due to the fact that $q$ is integer. 
\begin{proposition}\label{prop19-1}
	Let $f\in L^1([0,1])$ such that for almost all $x\in [0,1] $ (with respect to Lebesgue measure),
	\begin{equation}\label{e:f-fundamental}
	f(x)=q^{-1}\sum_{j=0}^{q-1}f((x+j)/q).
	\end{equation} 
	Then $f$ is almost everywhere a (complex) constant equal to $\int_0^1 f(x)\,\dd x$. 
\end{proposition}
In Section~\ref{s:Proof of prop19-1}, we construct remarkable examples of functions $f\not\in L^1([0,1])$ where \eqref{e:f-fundamental} is satisfied but in one case $f$ is not absolutely integrable and in another case $f$ is not measurable.

\subsection{Characterization of stationarity by the characteristic function of $X$}

Next, we characterize stationarity of $\{X_n\}_{n\geq 1}$ in terms of the characteristic function of $X$ given by
\[f(t):= \int \mathrm{e}^{itx} \, \dd F(x),\qquad t\in \R.\]
In particular, we discuss when $\dd F$ is a Rajchman measure, meaning that $f(t)\rightarrow 0$ as $t\to \pm\infty$, cf.\ Section~\ref{s:background}.


\begin{theorem}\label{fourier}
	\begin{enumerate}
		\item[(I)] Let $\tilde X$ be a random variable on $[0,1]$ with CDF  $\tilde{F}$ and  characteristic function $\tilde f$. Then $\tilde{F}$ satisfies \eqref{e:3} if and only if for all $k\in \mathbb{Z}$,
		\[\tilde f(2\pi k q)=\tilde f(2\pi k).\]
		
		\item[(II)] If $F$ satisfies \eqref{e:3},
		then  $\lim _{t \to \infty} f(t)$ exists if and only if there exists $c\in [0,1]$ such that for all $x\in [0,1]$, $F(x)=(1-c)x+c H(x)$. In this case, $c=\lim_{t\to \infty} f(t)$.  In particular, $\dd F$ is a Rajchman measure if and only if $F$ is the uniform CDF on $[0,1]$.
		
	\end{enumerate}
\end{theorem} 

By the Riemann-Lebesgue lemma, if $\dd F$ is absolutely continuous relative to Lebesgue measure, then it is also a Rajchman measure. Thus a corollary to Theorem~\ref{fourier}(II) is that if $F$ is a CDF satisfying \eqref{e:3}, then $\dd F$ is absolutely continuous with respect to Lebesgue measure if and only if $F$ is the uniform CDF on $[0,1]$. But a more direct argument for this is to apply Proposition~\ref{prop19-1}, see the beginning of Section~\ref{s:proofSecondThm}.

Salem in \cite{Salem1943} asked whether the measure $\dd ?$ corresponding to Minkowski's question-mark function restricted to $[0,1]$ is a Rajchman measure. It has recently been shown that $\dd ?$ is indeed a Rajchman measure \cite{JordanSahlsten,Persson}. Combining this fact with Theorem~\ref{fourier}(II) we obtain the following corollary.

\begin{corollary}\label{cor:questionmark}
	Minkowski's question-mark function does not satisfy \eqref{e:3} for any integer $q\geq 2$. {In particular, if Minkowski's question-mark function equals the CDF of a random variable as in \eqref{e:X}, then the stochastic process $\{X_n\}_{n \geq 1}$ cannot be stationary.}
\end{corollary}

\subsection{Characterization of stationarity by a decomposition result for $F$}\label{s:2.2}

\hspace{-3pt}The theorem below characterizes stationarity of $\{X_n\}_{n\geq 1}$ by properties of each part of the Lebesgue decomposition of $F$. 
It is a generalization of results obtained in \cite{Harris1955} and \cite{Dym1968};
our proof in Sections~\ref{s:proofSecondThm} is based on Theorem~\ref{HC1}, Proposition~\ref{prop19-1}, Theorem~\ref{fourier}, and a technical result (Lemma~\ref{lem:b2}).

We need the following notation and concepts. 

We call $s\in[0,1]$ a \emph{purely repeating base-$q$ number of order $n$} if the base-$q$ expansion of $s$ is of the form
\begin{equation}\label{may21}
s=(0.\overline{t_1\dots t_n})_q:=(0.t_1 \ldots t_n t_1\ldots t_n\ldots)_q 
=\sum_{j=1}^n t_jq^{-j}/\left(1-q^{-n}\right),
\end{equation}
where $n$ is {the smallest possible positive integer} and $t_1,\ldots,t_n\in\{0,\ldots,q-1\}$.
A purely repeating base-$q$ number cannot be a base-$q$ fraction; if $S_j(x):=(x+j)/q$ then the purely repeating number $(0.\overline{t_1\dots t_n})_q$ is the unique fixed point of the function $S_{t_1}\circ\dots \circ S_{t_n}$.
For any $n\in\mathbb N$ we call $(s_1,\ldots,s_n)$ a \emph{cycle of order $n\in\mathbb N$} if for some integers $t_1,\dots,t_n\in \{0,\ldots,q-1\}$,
\begin{equation}\label{e:sequence}
s_1=(0.\overline{t_1t_2\dots t_n})_q,\	s_2=(0.\overline{t_nt_1\dots t_{n-1}})_q,\ \dots,\ s_{n}=(0.\overline{t_2\dots t_nt_1})_q,
\end{equation}
and $s_1,\dots,s_n$ are pairwise distinct. Note that for two cycles $(s_1,\dots,s_n)$ and $(s_1',\dots,s_m')$, the sets { $\{s_1,\dots, s_n\}$ and $\{s_1',\dots, s_m'\}$} are either equal or disjoint. Table~\ref{tab:b1} shows the cycles up to the order of elements  
for $q=2,3$ and $n=1,2,3$. 
Moreover, let $H$ be the Heaviside function defined by $H(x)=0$ for $x<0$ and $H(x)=1$ for $x\ge0$. Finally, we say that $F$ is a mixture of an at most countable number of CDFs if there exist CDFs $F_1,F_2,\ldots$ and a discrete probability distribution $(\theta_1,\theta_2,\ldots)$ such that $\tilde F=\sum_{i}\theta_iF_i$.

\begin{table}
	\centering
	\bgroup
	\def\arraystretch{1.5}
	\begin{tabular}{l l l}
		&	$q=2$	&	$q=3$\\
		\hline
		$n=1$	& $ 0, 1 $& $ 0, \frac{1}{2}, 1 $\\
		\hline
		$n=2$	& $ (\frac{1}{3},\frac{2}{3}) $& $ (\frac{1}{8},\frac{3}{8}), (\frac{2}{8},\frac{6}{8}), (\frac{5}{8},\frac{7}{8}) $\\
		\hline
		$n=3$	& $ (\frac{1}{7},\frac{2}{7},\frac{4}{7}),(\frac{3}{7},\frac{5}{7},\frac{6}{7})$& \begin{tabular}[c]{@{}c@{}}
			$ (\frac{1}{26},\frac{3}{26},\frac{9}{26}), (\frac{2}{26},\frac{6}{26},\frac{18}{26}), (\frac{4}{26},\frac{12}{26},\frac{10}{26}),(\frac{5}{26},\frac{15}{26},\frac{19}{26}),$\\
			$(\frac{7}{26},\frac{21}{26},\frac{11}{26}),(\frac{8}{26},\frac{24}{26},\frac{20}{26}),(\frac{14}{26},\frac{16}{26},\frac{22}{26}),(\frac{17}{26},\frac{25}{26},\frac{23}{26})$
		\end{tabular}\\
	\end{tabular}
	\egroup
	\caption{All cycles up to equivalence when $q\in\{2,3\}$ and $n\in\{1,2,3\}$.}
	\label{tab:b1}
\end{table}

\begin{theorem}\label{thm1}
	$F$ satisfies the stationarity equation \eqref{e:3}	if and only if $F$ is a mixture of three CDFs $F_1,F_2,F_3$  whose corresponding probability distributions are mutually singular measures concentrated on $[0,1]$ {such that} $F_1,F_2,F_3$ satisfy the following statements {\rm (I)-(III)}: 
	\begin{enumerate} 
		\item[{\rm (I)}]
		$F_1$ is the uniform CDF on $[0,1]$, that is, $F_1(x)=x$ for $x\in[0,1]$. 
		
		\item[{\rm (II)}] $F_2$ is a mixture
		of an at most countable number of CDFs of the form
		\begin{equation}\label{e:F2type}
		F_{s_1,\ldots,s_n}(x):=\frac{1}{n}\sum_{j=1}^{n} H(x-s_j),\qquad x\in\mathbb R,
		\end{equation}
		where $(s_1,\ldots,s_n)$ is a cycle of order $n$.
	
		\item[{\rm (III)}]
		$F_3$ is singular continuous and satisfies \eqref{e:3}
		(with $F$ replaced by $F_3$). 
	\end{enumerate}
	Moreover, we have:
	\begin{enumerate}
		\item[{\rm (IV)}] $F_1$ and $F_2$ also satisfy \eqref{e:3} (with $F$ replaced by $F_1$ and $F_2$, respectively).  
	\end{enumerate}
\end{theorem}

The CDF 
$F_{s_1,\ldots,s_n}$ given by \eqref{e:F2type} is 
just the empirical CDF at the points in the cycle $(s_1,\ldots,s_n)$. 
Thus the following corollary follows immediately from \eqref{e:F2type}.

\begin{corollary}\label{cor2.4}
	Let $(s_1,\dots, s_n)$ be a cycle of order $n$, defined by $t_1,\ldots,t_n\in\{0,\ldots,q-1\}$ as given in \eqref{e:sequence} and assume that $X$ follows $F_{s_1,\ldots,s_n}$ given by \eqref{e:F2type}.
	\begin{enumerate} 
		\item[{\rm (I)}] If $n=1$, then $X_1=X_2=\ldots=t_1$ almost surely.		
		\item[{\rm (II)}] If $n\ge2$,
		then the distribution of $\{X_m\}_{m\ge1}$ is completely determined by the fact that
		$(X_1,\ldots,X_{n-1})$ is uniformly distributed on
		\begin{equation*}
		\{(t_1,\ldots,t_{n-1}),\, (t_n,t_1,\ldots,t_{n-2}),\, \ldots,\, (t_{2},\ldots,t_n)\},
		\end{equation*} 
		since almost surely $\{X_m\}_{m\ge1}$ is in a one-to-one correspondence to $X$ and 
		\begin{eqnarray*}
			(X_1,\ldots,X_{n-1})=(t_1,\ldots,t_{n-1})\ &\Rightarrow&\ X=s_1=(0.\overline{t_1\ldots t_n})_q\\ 
			(X_1,\ldots,X_{n-1})=(t_n,t_1,\ldots,t_{n-2})\ &\Rightarrow&\ X=s_2=(0.\overline{t_nt_1\ldots t_{n-1}})_q\\ 
			&\vdots&\\
			(X_1,\ldots,X_{n-1})=(t_{2},\ldots,t_n)\ &\Rightarrow&\  
			X=s_n=(0.\overline{t_2\ldots t_{n}t_1})_q.
		\end{eqnarray*}

	\end{enumerate}
\end{corollary}
\begin{remark}
	Corollary~\ref{cor2.4} shows that the stochastic process corresponding to a CDF as in (2.5) is a Markov chain of order $n-1$, but essentially it is equivalent to a uniform distribution on $n$ elements. Thus, a stationary stochastic process $\{X_n\}_{n\ge1}$ corresponding to a mixture of CDFs as in (2.5) will be rather trivial. Hence, by Theorem~\ref{thm1}, it only remains to understand those stationary stochastic processes $\{X_n\}_{n\ge1}$ which generate a singular continuous CDF $F$. This will be the topic of our follow up paper mentioned at the very end of Section 1.   
\end{remark}

\section{Proofs and further results}\label{s:proofs}

\subsection{Proof of Theorem~\ref{HC1}}\label{s:proof first thm}

Before proving Theorem~\ref{HC1}, we need the following two lemmas. 

\begin{lemma}\label{lem:horia1}
	Suppose that $\{X_n\}_{n\geq 1}$ is stationary. Then the probability of all $X_n$ having the same value starting from some $n_0>1$ and at least one $X_m$ having a different value for some $m<n_0$ is zero:
	\begin{equation}\label{horia1}
	\mathrm P\left(\bigcup_{0<m<n_0<\infty}\{X_m\not= X_{n_0}=X_{n_0+1}=\ldots\}\right)=0.
	\end{equation}
\end{lemma}

\begin{proof} It suffices to verify that for integers $0<m<n_0<\infty$,
	\begin{equation}\label{horia111}
	\mathrm P\left(X_m\not= X_{n_0}=X_{n_0+1}=\cdots\right)=0,
	\end{equation}
	where without loss of generality we may assume that $m=n_0-1$. 
{For any $k\in\{0,\ldots,q-1\}$, we have by the law of total probability for two events that
	\begin{align*}
	\mathrm P(X_{n_0}=X_{n_0+1}=\ldots=k)=&\,\mathrm P( X_{n_0-1}\neq k, X_{n_0}=X_{n_0+1}=\ldots=k) \\
	&+ \mathrm P(X_{n_0-1}=X_{n_0}=\ldots=k).
	\end{align*}
	Then \eqref{horia111} follows, since by stationarity of $\{X_n\}_{n\geq 1}$ we have
	\[\mathrm P(X_{n_0}=X_{n_0+1}=\ldots=k)=\mathrm P(X_{n_0-1}=X_{n_0}=\ldots=k).\] 
	Thereby \eqref{horia1} is verified.} 
\end{proof}

\begin{lemma}\label{lem:horia2}
	If $\tilde{F}$ is the CDF for a probability distribution on $[0,1]$ which obeys \eqref{e:3},
	then $\tilde{F}$ is continuous at { every $x \in \mathbb{Q}_q$}.
\end{lemma}
\begin{proof} Clearly, \eqref{e:3} is also true for $\tilde{F}(x)$ if $x=1$, so
	using \eqref{e:3} we have for any {$\delta \in \mathbb{Q}_q$} and for any base-$q$ fraction $x\in(\delta,1)$ or for $x=1$ that
	\begin{equation}\label{eq:discont}
	\tilde{F}(x)-\tilde{F}(x-\delta)=\sum_{j=0}^{q-1} [\tilde{F}((j+x)/q)-\tilde{F}((j+x)/q -\delta/q)].
	\end{equation}
	{Now, we prove the lemma by induction, considering first base $q$-fractions of order one.
	Letting $x=1$ gives
	\begin{align*}
	\tilde{F}(1)-\tilde{F}(1-\delta)=&\,\tilde{F}(1)-\tilde{F}(1-\delta/q)+\\
	&\,\sum_{j=0}^{q-2} [\tilde{F}((j+1)/q)-\tilde{F}((j+1)/q -\delta/q)],
	\end{align*} 
	and letting $\delta\downarrow 0$ we see that all the jumps of $\tilde{F}$ at $1/q,..., (q-1)/q$ must be zero, so $\tilde{F}$ is continuous at these points, which are first order base-$q$ fractions. Next let us assume that for a given order $n \geq 1$, $\tilde{F}$ is continuous at all base-$q$ fractions of order $n$. Let $x=(0.k_1k_2\ldots k_n)_q$ with $k_n \neq 0$. By using \eqref{eq:discont} and taking $\delta \downarrow 0$, we obtain that $\tilde{F}$ is continuous at all numbers of the form $(x+j)/q=(0.jk_1,\ldots k_n)_q$ for all $j \in \{0,\ldots,q-1\}$. This shows that $\tilde{F}$ is continuous at all base-$q$ fractions of order $n+1$. This completes the proof.} 
\end{proof}

{Now, to prove Theorem~\ref{HC1} we need the following notation and observations.} Let $(x_1,\ldots,x_n),$ $(t_1,\ldots,t_n)\in\{0,\ldots,q-1\}^n$. We write $(x_1,\ldots,x_n)\le(t_1,\ldots,t_n)$ if $\sum_{k=1}^n x_kq^{-k}\le\sum_{k=1}^n t_kq^{-k}$.
Define $t:=\sum_{j=1}^nt_jq^{-j}$. Note that  $x=\sum_{i=1}^\infty x_i q^{-i}\in [0,1]$ satisfies $x\leq t+q^{-n}$ if and only if one of the following two statements holds true: 
\begin{itemize}
	\item The first $n$ digits of $x$ obey $(x_1,\ldots,x_n)\le(t_1,\ldots,t_n)$ (regardless what the values of the next digits $x_{n+1}, x_{n+2},\ldots$ are).
	\item We have $x_1=t_1,\ldots,x_{n-1}=t_{n-1},x_n=t_n+1,x_{n+1}=x_{n+2}=\ldots=0$.
\end{itemize}
{
Define 
\begin{align*}
\mathcal{F}_1(t_1,\dots,t_n):=&\,\mathrm P((X_1,\ldots, X_n)\leq {(t_1,\dots,t_n)})\\
=&\,\sum_{(x_1,\ldots,x_n)\leq (t_1,\ldots,t_n)}\mathrm P(X_1=x_1,\dots,X_n=x_n).
\end{align*}
Stationarity of $\{X_n\}_{n\geq 1}$ is equivalent to that for every $(t_1,\ldots,t_n)\in \{0,\ldots,q-1\}^n$,  $\mathcal{F}_1(t_1,\ldots,t_n)$ equals  
\[\mathcal{F}_2(t_1,\dots,t_n):=\sum_{(x_2,\ldots,x_{n+1})\leq (t_1,\ldots,t_n)}\mathrm P(X_2=x_2,\dots,X_{n+1}=x_{n+1}),\]
cf.\ Kolmogorov's extension theorem.}
Furthermore, by \eqref{e:X}-\eqref{e:Fdef} we have
\begin{align}\label{e:Ftqn}
& F(t+q^{-n}) 
\nonumber 
=\mathcal{F}_1(t_1,\dots,t_n)+\\
&\mathrm P(X_1=t_1,\ldots,X_{n-1}=t_{n-1}, X_n=t_n+1,X_{n+1}=X_{n+2}=\ldots=0).
\end{align}

\begin{proof}[Proof of Theorem~\ref{HC1}(I)] 
	Assume that $\{X_n\}_{n\geq 1}$ is stationary. Using \eqref{e:Ftqn} together with Lemma~\ref{lem:horia1} we obtain 
	\begin{equation}\label{hdc2}
	F(t+q^{-n})=\mathcal{F}_1(t_1,\dots,t_n).
	\end{equation}
	Furthermore, stationarity of $\{X_n\}_{n\geq 1}$ implies that $X$ and the `left shifted' stochastic variable   $\sum_{n=1}^\infty X_{n+1}q^{-n}=qX-X_1$ are identically distributed. Thus,  
	\begin{equation*}
	F(x)=\mathrm P(qX-X_1\leq x)=\sum_{j=0}^{q-1}\mathrm P(X_1=j,\,X\le (x+j)/q).
	\end{equation*}
	We see that 
	\begin{align*}
	\mathrm P(X_1=0,\,X\le x/q)&=\mathrm P(X=0)+\mathrm P(0<\,X\le x/q)\\
	&=F(0)+(F(x/q)-F(0)).
	\end{align*}
	Further, for $j\in\{1,\ldots, q-1\}$,
	\begin{align*}
	\mathrm P(X_1=j,\,X\le (x+j)/q)=&\,\mathrm P(X_1=j, X_2=X_3=\ldots=0)\\
	&\,+ \mathrm P(j/q<\,X\le (x+j)/q)\\
	=&\,F((x+j)/q)-F(j/q),
	\end{align*}
	where we used \eqref{horia1} in order to get the second identity. This leads to \eqref{e:3}.

	Conversely, assume that $F$ satisfies \eqref{e:3}. Then, since $F$ is right continuous and {$\mathbb{Q}_q$} constitutes a dense subset of $[0,1]$, \eqref{e:3} holds for all $x\in[0,1]$. 
{ Further, by Kolmogorov's extension theorem we can define  the distribution of $\{X_n\}_{n\geq 1}$ on $\{0,\ldots,q-1\}$  by specifying the finite dimensional distribution $\mathcal{F}_1$ of $( X_1,\dots,  X_n)$ for every integer $n\ge1$ in a consistent way, setting
\begin{equation}\label{eq:K1}
    \mathcal{F}_1(t_1,\dots,t_n)=\mathrm{P}((X_1,\dots,X_n)\leq (t_1,\dots,t_n)):=F(t+q^{-n}).
\end{equation}
}
		Furthermore, for any $\epsilon>0$ and {$x_1,\dots, x_n\in \{0,\dots,q-1\}$}, the inequality
	\begin{align}\label{eq:lang}
	&\mathrm{P}(X_1=x_1,\ldots, X_n=x_n,X_{n+1}=\ldots=0)\nonumber\\&\qquad+\mathrm P(X_1=x_1,\ldots,X_{n-1}=x_{n-1}, X_n=x_n-1,X_{n+1}=\ldots=q-1)\nonumber\\
	&\quad\leq \mathrm P(x-\epsilon<X\leq x)=F(x)-F(x-\epsilon)
	\end{align}
	holds and hence by Lemma~\ref{lem:horia2}  the probability of realizing a base-$q$ fraction is 0. Hence the second term in the right hand side of \eqref{e:Ftqn} is 0, and so \eqref{hdc2} holds again true.
	Consequently, for any $n\in\mathbb N$ and $(t_1,\ldots,t_n)\in\{0,\ldots,q-1\}^n$, 
	\begin{align*}
	\mathcal{F}_2(t_1,\dots,t_n)&=\sum_{j=0}^{q-1}\sum_{(x_2,\ldots,x_{n+1})\leq (t_1,\ldots,t_n)}\mathrm P(X_1=j,X_2=x_2,\dots,X_{n+1}=x_{n+1}) \\
	&=\mathrm P(X=0)+\sum_{j=0}^{q-1} \mathrm P(j/q<X\leq j/q + t/q +q^{-n-1})\\
	&=F(0)+\sum_{j=0}^{q-1} (F((t+q^{-n}+j)/q)-F(j/q)) \\
	&=F(t+q^{-n}), 
	\end{align*}
	using in the first identity the law of total probability, in the second that the probability of realizing a base-$q$ point is zero, in the third \eqref{e:Fdef}, and in the last
	\eqref{e:3}.
	Thereby \eqref{eq:K1} gives that $\mathcal{F}_1(t_1,\dots,t_n)=\mathcal{F}_2(t_1,\dots,t_n)$ for every $n\in\mathbb N$ and every $(t_1,\ldots,t_n)\in\{0,\ldots,q-1\}^n$, so
	$\{X_n\}_{n\geq 1}$ is stationary.
\end{proof}

\begin{proof}[Proof of Theorem~\ref{HC1}(II)] Let $\phi(x)= \{x_n\}_{n\ge1}$ be the one-to-one mapping on $[0,1]\setminus\mathbb Q_q$ corresponding to mapping $x$ into its base-$q$ digits $x_1,x_2,\ldots$, that is, $x=\sum_{n=1}^\infty x_nq^{-n}$.   
	Further, let $\tilde F$ be a CDF for a random variable $\tilde X$ on $[0,1]$ 
	such that $\tilde F$ satisfies \eqref{e:3} for all $x\in \mathbb Q_q$. By Lemma~\ref{lem:horia2} and since $\mathbb Q_q$ is countable, we can assume that $\tilde X\not\in\mathbb Q_q$. Then $\tilde X$ is in a one-to-one correspondence to $\{\tilde X_n\}_{n\ge1}:=\phi(\tilde X)$ and $\tilde X=\sum_{n=1}^\infty \tilde X_nq^{-n}$ follows $\tilde F$. We conclude from Theorem~\ref{HC1}(I) that the stochastic process $\{\tilde X_n\}_{n\ge1}$ is stationary. Since the distribution of $\{\tilde X_n\}_{n\ge1}$ is induced by that of $\tilde X$ and the one-to-one mapping $\phi$, let us show that $\{\tilde X_n\}_{n\ge1}$ is the unique (up to its distribution) stationary stochastic process on $\{0,\ldots,q-1\}$ so that $\sum_{n=1}^\infty \tilde X_nq^{-n}$ follows $\tilde F$: if $\{\bar X_n\}_{n\ge1}$ is another stationary stochastic process on $\{0,\ldots,q-1\}$ so that $\sum_{n=1}^\infty \bar X_nq^{-n}$ follows $\tilde F$, then for any event $G$ of sequences $\{x_n\}_{n\ge1}$ so that each $x_n\in\{0,\ldots,q-1\}$ and $\sum_{n=1}^\infty x_nq^{-n}\not\in\mathbb Q_q$, we have
	\[\mathrm P(\{\bar X_n\}_{n\ge1}\in G)=\mathrm P(\tilde X\in\phi^{-1}(G))=\mathrm P(\{\tilde X_n\}_{n\ge1}\in G).\]
	Furthermore, by right continuity of $\tilde F$ and since {$\mathbb{Q}_q$} is dense on $(0,1)$, $\tilde F$ satisfies \eqref{e:3} for all $x\in(0,1)$. Finally, $\tilde F$ obviously satisfies \eqref{e:3} for $x=0$ and $x=1$.
\end{proof}

	\subsection{Proof of Proposition~\ref{prop:indep}}
	Suppose that the $X_n$'s are IID. Since 
	\eqref{eq:indep} holds for $x=1$, $F$ is right continuous, and {$\mathbb{Q}_q$} is dense on $(0,1)$, it suffices to show that \eqref{eq:indep} holds for all { $x \in \mathbb{Q}_q$}. {Let $x=(0.x_1\dots x_n)_q$ with $x_1,\dots,x_n\in \{0,\dots,q-1\}$ where $x_n \neq 0$}. By Lemma~\ref{lem:horia1}, 
	\begin{equation*}
	    \mathrm{P}(X_1=x_1,\dots, X_n=x_n, X_{n+1}=X_{n+2}=\ldots=0)=0.
	\end{equation*}
	Hence, setting $\sum_{j=0}^{-1} \dots =0$,
	\begin{align*}
	    F(x)&=\sum_{j=0}^{x_1-1}\rmP(X_1=j)+\rmP(X_1=x_1)\sum_{j=0}^{x_2-1}\rmP(X_2=j)\\
	    &\quad+\dots+\rmP(X_1=x_1)\sum_{j=0}^{x_n-1} \rmP(X_2=x_2,\dots, X_{n-1}=x_{n-1},X_n=j) \\
	   &=\sum_{j=0}^{x_1-1} \rmP(X_1=j)+\rmP(X_1=x_1)F(qx-x_1),
	\end{align*}
    using that the $X_n's$ are independent in the first identity and identically distributed in the second identity.
    Furthermore, for $j\in \{0,\dots,q-1\}$,
{
\begin{equation}\label{eq:benjamin12321}
	F(qx-j)=\begin{cases}
0 & \textup{if } x_1<j,\\
1 & \textup{if } x_1>j, \\
F((0.x_2\ldots,x_n)_q) & \textup{if } x_1=j,
\end{cases}
\end{equation}
}
	and so we conclude that \eqref{eq:indep} holds at $x$.
	
	Conversely, suppose that \eqref{eq:indep} holds. Then $F(0)=\rmP(X_1=0)F(0)$, so either $\rmP(X_1=0)=1$ or $F(0)=0$. In the first case, since by stationarity the $X_n$'s are identically distributed, we obtain immediately that the $X_n$'s are IID.
	So assume that $F(0)=0$ and let $x=(0.x_1\dots x_n)_q$ with $x_1,\dots,x_n\in \{0,\dots,q-1\}$. Then
	\begin{equation*}
	    \rmP(X_1=x_1,\dots, X_n=x_n)=\rmP(x\leq X\leq x+q^{-n})=F(x+q^{-n})-F(x),
	   \end{equation*}
	    using in the last identity that $\rmP(X=x)=0$, cf.\ Lemma~\ref{lem:horia1}.
{
Thus,
	    \begin{align*}
	    \rmP(X_1=x_1,\dots, &X_n=x_n)=F(x+q^{-n}) - F(x) \\
	    &=\sum_{j=0}^{q-1}\rmP(X_1=j) \Big (F(q(x+q^{-n})-j)-F(qx-j)\Big )\\
	    &=\rmP(X_1=x_1) \Big (F((0.x_2\dots x_n)_q+q^{-n+1})-F((0.x_2\dots x_n)_q)\Big ),
	    \end{align*}
	    where in the second equality we use \eqref{eq:indep} and in the third equality we use \eqref{eq:benjamin12321}. The dependence on $x_1$ has now been isolated in the factor $\rmP(X_1=x_1)$, and using the same method for the other variables $x_2,\ldots,x_n$ we obtain
	    \begin{align*}
	    \rmP(&X_1=x_1,\dots,X_n=x_n)\\
	    &=\rmP(X_1=x_1)\sum_{j=0}^{q-1}\rmP(X_1=j) \Big (F(x_2-j+(0.x_3\dots x_n)_q+q^{-n+2})\\
	    &\phantom{=\rmP(X_1=x_1)\sum\rmP(X_1=j) \Big (}-F(x_2-j+(0.x_3\dots x_n)_q)\Big )\\	    &=\rmP(X_1=x_1)\rmP(X_1=x_2) \Big (F((0.x_3\dots x_n)_q+q^{-n+2})-F((0.x_3\dots x_n)_q)\Big )\\
	    &\;\,\vdots\\
	    &=\rmP(X_1=x_1)\cdots\rmP(X_1=x_{n-1})\sum_{j=0}^{q-1}\rmP(X_1=j) \Big (F(x_n-j+1)-F(x_n-j)\Big )\\	    
	    &=\rmP(X_1=x_1)\cdots \rmP(X_{1}=x_{n-1})\Big (\rmP(X_1=x_n)(1-F(0))+\rmP(X_1=x_{n}+1)F(0)\Big ).	
	    \end{align*}	

}	    Consequently, since $F(0)=0$ and by stationarity the $X_n$'s are identically distributed,
	\[\rmP(X_1=x_1,\dots, X_n=x_n)=\rmP(X_1=x_1)\cdots \rmP(X_{1}=x_{n})=\rmP(X_n=x_1)\cdots \rmP(X_{n}=x_{n}).\]
	Therefore, the $X_n$'s are IID.

\subsection{Proof of Proposition~\ref{prop19-1} and related examples}\label{s:Proof of prop19-1}

\begin{proof}[Proof of Proposition~\ref{prop19-1}] 
{ We first claim that for every $n\in\mathbb N$, there exists a Borel set $A_n\subseteq [0,1]$ of Lebesgue measure 1 such that for all $x\in A_n$,
	\begin{equation}\label{19-1}
	f(x)=q^{-n}\sum_{j=0}^{q^n-1} f((x+j)/q^n).
	\end{equation}
	The case $n=1$ follows by the assumption in Proposition~\ref{prop19-1}.
Define 
\begin{equation*}
\tilde{A}_2\coloneqq \bigcup_{j=0}^{q-1} \frac{A_1+j}{q},\qquad A_2\coloneqq A_1 \cap \tilde{A}_2. 
\end{equation*}	
Since the Lebesgue measure of $ \frac{A_1+j}{q} \cap  \frac{A_1+k}{q}$ is $0$ for $j \neq k$, $\tilde{A}_2$ and hence $A_2$ have Lebesgue measure $1$. For every $x\in A_2$, we have $(x+j)/q\in A_1$ and hence using \eqref{e:f-fundamental} for $f((x+j)/q)$ we obtain
\begin{align*}
f(x)=q^{-2} \sum_{j=0}^{q-1} \sum_{k=0}^{q-1} f\Big( \frac{x+j}{q^2} + \frac{k}{q} \Big) = q^{-2} \sum_{j=0}^{q^2-1} f\Big( \frac{x+j}{q^2}\Big).
\end{align*}
Continuing in this way the claim is verified.
}	
	Define $\mathcal{I}:=\int_0^1 f(t)\dd t\in \mathbb{C}$ and $I_{j,n}:=(jq^{-n},jq^{-n}+q^{-n})$. Then \eqref{19-1} gives for all $x\in A_n$,
	\begin{equation*}
	f(x)-\mathcal{I}=\sum_{j=0}^{q^n-1}  \int_{I_{j,n}}\left \{ f((x+j)/q^n) -f(y)\right \}\,\dd y,
	\end{equation*}
	so 
	\begin{equation*}
	|f(x)-\mathcal{I}|\leq \sum_{j=0}^{q^n-1}  \int_{I_{j,n}}\left \vert f((x+j)/q^n) -f(y) \right \vert \,\dd y.
	\end{equation*}
	Integrating with respect to $x\in[0,1]$ and making the change of variable 
	$t =(x+j)/q^n\in I_{j,n}$, 
	we obtain
	\begin{equation}\label{19-3}
	\int_0^1|f(x)-\mathcal{I}|\,\dd x\leq q^n\sum_{j=0}^{q^n-1}  \int_{I_{j,n}}\int_{I_{j,n}}|f(t) -f(y)|\,\dd y\,\dd t.
	\end{equation}
	Now, for any $\varepsilon>0$, there exists a uniformly continuous function $g_\varepsilon:[0,1]\to\mathbb R$ such that $\Vert f-g_\varepsilon\Vert_{L^1}\leq \varepsilon /3$, where we consider the usual $L^1$-norm. Writing 
	\[|f(t)-f(y)|\leq |f(t)-g_\varepsilon(t)|+| f(y)-g_\varepsilon(y)|+|g_\varepsilon (t)-g_\varepsilon(y)|,\]
	we obtain from \eqref{19-3} that 
	\[\int_0^1|f(x)-\mathcal{I}|\,\dd x\leq 2\varepsilon/3 +\sup_{|t-y|\leq q^{-n}} |g_\varepsilon (t)-g_\varepsilon(y)|.\]
	Since $g_\varepsilon$ is uniformly continuous, for any sufficiently large $n=n(\varepsilon)\in\mathbb N$, we have
	\[\sup_{|t-y|\leq q^{-n(\varepsilon)}} |g_\varepsilon (t)-g_\varepsilon(y)|\leq \varepsilon/3.\]
	Thus $\int_0^1|f(x)-\mathcal{I}|\,\dd x\leq \varepsilon$. Consequently, $f=\mathcal{I}$ almost everywhere on $[0,1]$. 
\end{proof}

We end this section by presenting two examples of functions $f\not\in L^1([0,1])$ where \eqref{e:f-fundamental} is satisfied but in one case $f$ is not absolutely integrable and in another case $f$ is not measurable.

\begin{example}[A solution to \eqref{e:f-fundamental} which is not in {$L^1([0,1])$}]\label{ex:7}
	Let $q=2$. Below we construct a solution to \eqref{e:f-fundamental} which is piecewise smooth on $[0,1]$, has finite jumps at all dyadic fractions $1-2^{-n}$ with $n\in\mathbb N$, but whose integral diverges. 
	
	For this we notice the following. 
	For any function $f:[0,1]\rightarrow\mathbb R$, define $g(x):=f(x)-1/(1-x)$ for $x\in[0,1)$, and let $g(1)$ be any number. 
	Then $f$ satisfies \eqref{e:f-fundamental} if and only if $g$ satisfies the equation
	\begin{equation}\label{e:g} 
	g(x)=g(x/2)/2 +g((x+1)/2)/2 +1/(2-x)\qquad\mbox{for almost all $x\in[0,1]$}.
	\end{equation}

	To construct a particular solution $g$ to \eqref{e:g}, we start by setting $g(x):=0$ for all $x\in [0,\frac12)$. Then $g(x/2)=0$ for all $x\in [0,1)$, and in accordance with \eqref{e:g} we should have 
	\begin{equation}\label{e:g2}
	g((x+1)/2)=2g(x)-2/(2-x)\qquad\mbox{for all $x\in[0,1)$},
	\end{equation}
	which is possible because of the following observations. 
	For each $n\in\mathbb N\cup\{0\}$, the map 
	\[[1-2^{-n},1-2^{-n-1})\ni {x\mapsto (x+1)/2}\in [1-2^{-n-1},1-2^{-n-2})\]
	is a bijection. Thus, for $n=1,2,\ldots$, we can inductively use \eqref{e:g2} to compute $g(x)$ for all $x\in [1-2^{-n},1-2^{-n-1})$.  We see that $g$ becomes more and more negative near $1$. Now, the function $f(x)=g(x)+1/(1-x)$ is not constant on $[0,1)$, since $f(x)=1/(1-x)$ on $[0,\frac12)$. Although $f$ satisfies \eqref{e:f-fundamental} and is smooth on each interval $[1-2^{-n},1-2^{-n-1})$ with $n\in\mathbb N$, $f$ cannot have a finite integral due to Proposition~\ref{prop19-1}. 
	Figure~\ref{fig:f_nonintegrable} shows a plot of $f$.	
\end{example}

\begin{figure}[hbt!]
	\centering
	\includegraphics[width=0.5\textwidth]{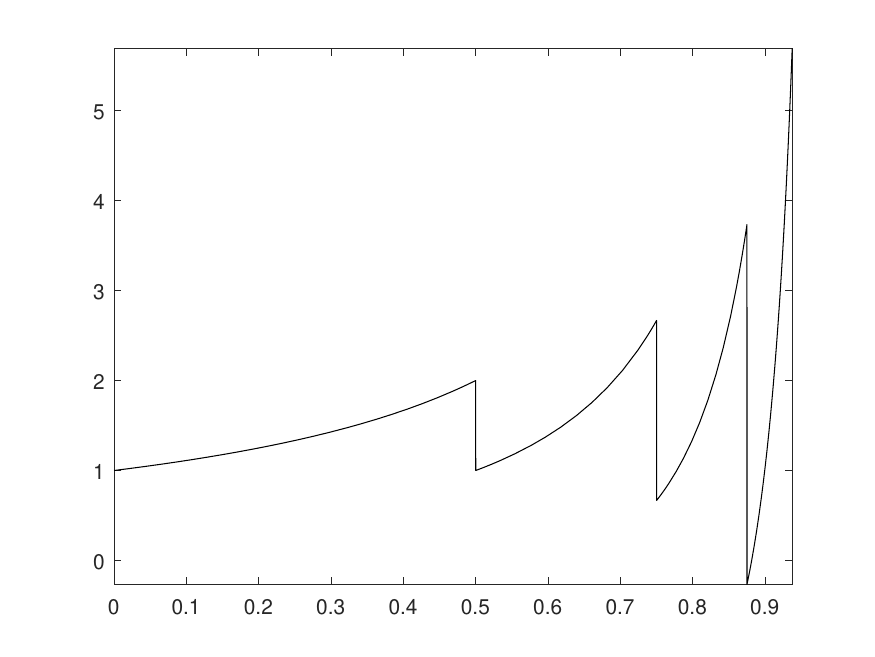}
	\caption{The function  $f(x)=g(x)+1/(1-x)$ for $x\in[0,1-2^{-4}]$, where $g$ is given as in Example~\ref{ex:7}.}
	\label{fig:f_nonintegrable}
\end{figure}

\begin{example}[A non-measurable solution to \eqref{e:f-fundamental}]\label{ex:8}
	Let $q=2$. Below we construct a solution to \eqref{e:f-fundamental}, which is bounded but cannot be measurable. 
	
	Define $G:=\{(2^n,r)\,|\,n\in\mathbb Z,\,r\in\mathbb D\}$, where 
	$\mathbb D:=\{{m}2^{-n}\,|\,m\in\mathbb Z,\,n\in\mathbb N\}$. Then $G$ is a group with product
	\[(2^m,p)(2^n,r)=(2^{m+n},p+2^mr)\qquad\mbox{for }(2^m,p),(2^n,r)\in G,\]
	and $G$ 
	acts on $\mathbb R$ by 
	\[(2^n,r)x:=2^nx+r\qquad\mbox{for }(2^n,r)\in G,\ x\in\mathbb R.\] 
	For any $x\in [0,1]$, let $M_x$ be the restriction of the orbit $\{2^nx+r\,|\,n\in\mathbb Z,\,r\in\mathbb D\}$ to $[0,1]$ (this is a countable dense set in $[0,1]$).  Given any $x\in [0,1]$, then both $x/2$ and $(x+1)/2$ belong to $M_x$. 
	
	By the axiom of choice, given any orbit restriction $M\cap [0,1]$, we can pick a representative $C(M)\in [0,1]$ and thus construct a function 
	\[f(x):=C(M_x)\in [0,1],\qquad x\in [0,1].\]
	{ Then $f$ is constant on each orbit, though with different constants on different orbits. Since $x$, $x/2$, and $(x+1)/2$ always belong to the same orbit, \eqref{e:f-fundamental} is satisfied everywhere. We now show that this (bounded) function cannot be measurable due to Proposition~\ref{prop19-1}. Indeed, if $f$ was measurable, it would be integrable and equal to a constant on $[0,1]$ outside some set $A$ of zero Lebesgue measure. Since $f$ has different (constant) values on different orbits and because $f$ is constant on $[0,1]\setminus A$, it implies that $[0,1]\setminus A$ is a subset of exactly one orbit. As the orbit is countable and has Lebesgue measure zero, it follows the Lebesgue measure of $[0,1]$ is zero. Hence we have a contradiction.  }
\end{example}

\subsection{Proof of Theorem \ref{fourier}}
First, we prove Theorem~\ref{fourier}(I) with $\tilde X$, $\tilde{F}$, and  $\tilde{f}$ as in the theorem and {$S_j(x)=(x+j)/q$ as in Section~\ref{s:2.2}}. Note that for any $k\in\mathbb{Z}$,
\begin{align}
\tilde{f}(2\pi k q)&=\int_0^1 \exp(2\pi i  x kq)\, \dd \tilde F(x)\nonumber  =\sum_{j=0}^{q-1}\int_{j/q}^{(j+1)/q} \exp(2\pi i x k q)\, \dd \tilde F(x)\\\nonumber 
&=\sum_{j=0}^{q-1}\int_{0}^{1} \exp(2\pi i (x+j) k)\, \dd (\tilde F\circ S_j)(x)\\ 
&=\int_{0}^{1} \exp(2\pi i x k) \,\dd (\sum_{j=0}^{q-1}\tilde F\circ S_j)(x). \label{may17'}
\end{align}
If $\tilde F$ satisfies \eqref{e:3} then $\dd \tilde F=\sum_{j=0}^{q-1} \dd (\tilde F\circ S_j)$, which combined with \eqref{may17'} shows that 
\begin{equation}\label{eq:tildef}
\tilde f(2\pi kq)= \tilde f(2\pi k)
\end{equation}
for all $k\in \mathbb Z$.

Now, suppose that 
\eqref{eq:tildef} holds for all $k\in \mathbb{Z}$. Define $\dd \tilde G:=\dd (\sum_{j=0}^{q-1} \tilde F\circ S_j)$ and $\tilde g(t):=\int \mathrm{e}^{ixt} \dd \tilde G(x)$.  From \eqref{may17'} we have that $\tilde{f}(2\pi kq)=\tilde{g}(2\pi k)$  which together with \eqref{eq:tildef} implies $\tilde f(2\pi k)=\tilde g(2\pi k)$ for all $k\in \mathbb{Z}$. Recall that any continuous $\mathbb{Z}$-periodic function $\varphi\colon \R\to \mathbb{C}$ is a uniform limit of trigonometric polynomials $\sum_{k=- N}^N c_k^N\mathrm{e}^{2\pi i kx}$, where each $c_k^N\in \mathbb{C}$ and $N\in \mathbb{N}$ \cite{BabyRudin}.  Taking the limit $N\to \infty$ and using Lebesgue's dominated convergence theorem we get 
\begin{equation}\label{eq:zperiod}
\int_{[0,1]} \varphi(x)\,\dd \tilde F(x)=\int_{[0,1]} \varphi(x) \,\dd \tilde G(x).
\end{equation}
The remaining part of this proof consists of verifying  \eqref{eq:zperiod} when $\varphi$ is merely continuous and then applying the Riesz-Markov theorem{, which implies that a positive linear functional on $C_0([0,1])$ can be represented by a unique measure (see \cite{Riesz1909})}. First, we establish equality of $\dd \tilde F$ and $\dd \tilde G$ at the endpoints of the interval $[0,1]$. Since the indicator function of $\mathbb{Z}$ can be  pointwise approximated by uniformly bounded continuous $\mathbb{Z}$-periodic functions, another application of Lebesgue's dominated convergence theorem in \eqref{eq:zperiod} gives
\[\dd\tilde F(\{0\})+\dd\tilde F(\{1\})=\dd\tilde G(\{0\})+\dd\tilde G(\{1\}),\]
where by definition of $\mathrm d\tilde G$,
\[\dd\tilde G(\{0\})=\dd \tilde F(\{0\}+\dd \tilde F(\{1/q\})+\dots+\dd \tilde F(\{(q-1)/q\})\]
and
\[\dd\tilde G(\{1\})=\dd \tilde F(\{1/q\})+\dots+\dd \tilde F(\{(q-1)/q\})+\dd\tilde F(\{1\}).\]
From this it immediately follows that  $\dd \tilde{F}(\{j/q\})=0$ for  $j=1,\ldots, q-1$, which leads to $\dd\tilde F(\{0\})=\dd \tilde G(\{0\})$ and $\dd\tilde F(\{1\})=\dd \tilde G(\{1\})$. Second, we extend \eqref{eq:zperiod} to all continuous functions on $[0,1]$ in the following way. If $\psi\colon [0,1]\to \mathbb{C}$ is continuous, define for $n=1,2,\dots$,
{
\[\varphi_n(x):= \begin{cases}
	\psi(x)& \textup{ for } x\in [0,1-\frac 1 n],\\
	n[\psi(0)-\psi(1-\frac{1}{n})]\Big(x-(1-\frac 1 n)\Big)+\psi(1-\frac{1}{n})& \textup{ for } x\in (1-\frac{1}{n},1].
	\end{cases}
	\] 
}
This is a uniformly bounded sequence of continuous and $\mathbb{Z}$-periodic functions converging pointwise to $\psi$ on $[0,1)$. Since $\dd \tilde F(\{1\})=\dd \tilde G(\{1\})$, it follows from Lebesgue's dominated convergence theorem that
\begin{align*}
\int_{[0,1]} \psi(x) \,\dd \tilde F(x)&= \dd \tilde F(\{1\})\psi(1)+\lim_{n\to \infty} \int_{[0,1)}\varphi_n(x) \,\dd \tilde F(x)\\
&=\dd \tilde G(\{1\})\psi(1)+\lim_{n\to \infty} \int_{[0,1)}\varphi_n(x) \,\dd \tilde G(x)\\
&=\int_{[0,1]} \psi(x) \,\dd \tilde G(x).
\end{align*}
{
The maps
\begin{align*}
C_0([0,1])\ni\psi \mapsto \int_{[0,1]} \psi(x) \mathrm{d}\tilde{F}(x) \in \mathbb{C}, \quad C_0([0,1])\ni\psi \mapsto \int_{[0,1]} \psi(x) \mathrm{d}\tilde{G}(x) \in \mathbb{C}
\end{align*}
are positive linear functionals and can be represented by a unique measure according to Riesz-Markov theorem (see \cite{Riesz1909}), thus $\dd\tilde F=\dd \tilde G$ on $[0,1]$.} Equivalently $\tilde F$ satisfies \eqref{e:3} and the proof of Theorem~\ref{fourier}(I) is complete.

Next, we prove Theorem~\ref{fourier}(II). As the `if' part of the proof follows from a direct calculation, we only prove that if $c:=\lim_{t\to \infty}f(t)\in \mathbb{C}$ exists, then $0\leq c\leq 1$ and for every $x\in[0,1]$,
\begin{equation}\label{e:Kasper}
F(x)=(1-c)x+cH(x)
\end{equation} 
where $H$ is the Heaviside function.
Since $\lim_{t \to \infty} f(t) = c$, for any $a\in \R$ a straightforward calculation gives 
\begin{equation*}
\lim_{T \to \infty} T^{-1}\int_0^T f(t)\mathrm{e}^{-ita}\,\dd t =\left\{
\begin{matrix}
c\quad {\rm if}\; a=0, \\
0\quad {\rm if}\; a\neq 0.
\end{matrix}\right . 
\end{equation*}
{
Define 
\begin{align*}
f_T (x) \coloneqq \frac{1}{T}\int_0^T \mathrm e^{it(x-a)} \mathrm{d}t = \begin{cases} 1 & \textup{ if } x=a, \\
\frac{1}{i(x-a)T}(\mathrm e^{iT(x-a)}-1) & \textup{ if } x \neq a.
\end{cases}
\end{align*}
Using Fubini's theorem we obtain
\begin{align*}
\frac{1}{T}\int_0^T f(t)\mathrm e^{-ita}\dd t = \int_{[0,1]} f_T(x) \dd F(x).
\end{align*}
Since $\vert f_T(x)\vert \leq 1$ and $f_T(x)$ converges pointwise to the indicator function of the set $\{a\}$, an application of Lebesgue's dominated convergence theorem shows that the above limit equals $\dd F(\{a\})$. Consequently, $c=F(0)\in[0,1]$ and $F$ is continuous on $(0,1)$. 

It remains to verify \eqref{e:Kasper}. If $c=1$, then $F=H$ and \eqref{e:Kasper} follows. }
Assume $c<1$. Then 
\[G(x):=(F(x) - cH(x))/(1-c)\]
is a continuous CDF that also satisfies the stationarity condition \eqref{e:3}. Thus, defining $g(t):=\int \mathrm{e}^{itx}\dd G(x)$, we obtain the identity $g(2\pi kq)=g(2\pi k)$ for all $k\in \mathbb{Z}$. A repeated use of this identity gives for any $m\in\mathbb N$ and $k\in\mathbb{Z}$ that
\begin{equation}\label{juni20}
g(2\pi k q^{m})=g(2\pi k).
\end{equation}
From the definition of $G$ it follows that $\lim_{t \to \infty} g(t) = 0${,} and since $c$ is real we also obtain
$\lim_{t \to -\infty} g(t) = 0$.
Combining this with \eqref{juni20} where we take $m$ to infinity we conclude that $g(2\pi k)=0$ for all $k\in\mathbb{Z}\setminus\{ 0\}$.
If $h(t):=\int_0^1 \mathrm{e}^{itx}\dd x$ then also $h(2\pi k)=0$ for all $k\in \mathbb{Z}\setminus\{0\}$, and since $\dd G(\{1\})=0$ it follows from the same arguments as in the proof of Theorem~\ref{fourier}(I) that $\dd G=\dd x$ on $[0,1]$.
Consequently, $F(x)=(1-c)x+cH(x)$ for all $x\in[0,1]$.

\subsection{Proof of Theorem~\ref{thm1}}\label{s:proofSecondThm}
The Lebesgue-Radon-Nikodym theorem \cite{Folland,Billingsley1995} leads to the decomposition $F(x) = \theta_1 F_1(x) + \theta_2 F_2(x) + \theta_3 F_3(x)$ for $x\in [0,1]$, where $\theta_1,\theta_2,\theta_3\geq 0$ and $\theta_1 + \theta_2 + \theta_3 = 1$, $F_1$ is an absolutely continuous CDF on $[0,1]$, $F_2$ is a discrete CDF on $[0,1]$, and $F_3$ is singular continuous CDF on $[0,1]$.
Proposition~\ref{prop19-1} implies that $F_1' =1$ almost everywhere on $[0,1]$, and so since $F_1$ is absolutely continuous, $F_1(x)=x$ for all $x\in [0,1]$. Thus $F_1$ satisfies \eqref{e:3} and it only remains to show that $F_2$ is as claimed in Theorem~\ref{thm1}(II) and satisfies \eqref{e:3}.

\subsubsection{Proof of Theorem~\ref{thm1}(II)}\label{s:5.1.2}

\begin{lemma}\label{lem:b2}
	Suppose that $F$ satisfies \eqref{e:3} and $s\in[0,1]$ is a discontinuity of $F$. Then there exists $n\in\mathbb{N}$ and a cycle $(s_1,\ldots,s_n)$ in the sense of \eqref{e:sequence} such that $s=s_1$ and $s_1,\ldots,s_n$ are discontinuities of $F$. Furthermore, the jumps of $F$ at these $n$ discontinuities are all equal.
\end{lemma}

\begin{proof} We start by investigating what can happen at $0$ and $1$. Both $0=(0.\overline{0})_q$ and $1=(0.\overline{q-1})_q$ are purely repeating base-$q$ numbers of order 1 and they can be discontinuity points because both $H(x)$ and $H(x-1)$ satisfy \eqref{e:3}. Therefore in the following, we will only consider possible discontinuities at $x\in (0,1)$. {Our idea is then to show that each point of discontinuity belongs to a `cycle' of finitely many points which are all 
discontinuities and the jump at each point is the same. }
	
	As in Section~\ref{s:2.2}, define $S_j(x):=(x+j)/q$ for $x\in (0,1)$ and $j=0,\dots,q-1$. Then, by Theorem~\ref{HC1}, it follows that for any $x\in (0,1)$, there exists a sufficiently small $\delta_0>0$ such that for all $\delta\in (0,\delta_0)$,
	\begin{equation}\label{eq:ben123}
	F(x)-F(x-\delta)=\sum_{j=0}^{q-1}[F(S_j(x))-F(S_j(x-\delta))].
	\end{equation}
	For $x\in (0,1)$, define $L_0(x):=\{x\}$ and $L_n(x):=\bigcup_{j=0}^{q-1}S_j(L_{n-1}(x))$, $n=1,2,\ldots$. Furthermore, let $J_x:=\lim_{\delta\downarrow 0} [F(x)-F(x-\delta)]$ denote the jump of $F$ at $x\in (0,1)$. Taking $\delta\downarrow 0$ in \eqref{eq:ben123} shows that 
	\begin{equation}\label{eq:8-6-2}
	J_x=\sum_{y\in L_1(x)} J_y.
	\end{equation} 
	Suppose $s\in (0,1)$ is a discontinuity of $F$ with jump $J_s>0$ and let $k>1/J_s$ be an integer. First, we show that the sets $L_0(s), \dots, L_k(s)$ are not pairwise disjoint. For the purpose of a contradiction assume that $L_0(s), \dots, L_k(s)$ are pairwise disjoint. By assumption $s\not \in L_1(s)$, thus replacing $x=s$ in \eqref{eq:8-6-2} shows that $F$ has a total jump of at least $2J_s$: one $J_s$ from $s$, and the other $J_s$ from the accumulated contribution of all the points of $L_1(s)$. {Using \eqref{eq:8-6-2} for $x=S_j(s)$ with $j=0,\ldots,q-1$, we see that the possible jump at each $S_j(s)$ equals the total accumulated jump at the points of $L_1(S_j(s))$.} Hence, by assumption the total jump of $F$ at the points of $L_2(s)$ is again $J_s$. Continuing this way we obtain that $\sum_{x\in L_j(s)} J_x =J_s$ for $j=0,1,\dots, k$. By the choice of $k$ this contradicts $F\leq 1$ and hence the sets $L_0(s), \dots, L_k(s)$ are not pairwise disjoint.
	
	Next, let $n$ denote the smallest integer (not necessarily larger than $1/J_s$) such that $L_0(s),\dots,L_n(s)$ are not pairwise disjoint. We will show that $s\in L_n(s)$. Suppose this is not the case. Then by the choice of $n$ there exist an integer $ m$ with $1\leq m<n$ and $j_1,\dots,j_{m},j_1',\dots,j_{n}'\in \{0,\dots,q-1\}$ such that 
	\begin{equation}\label{eq:8-6-1}
	S_{j_1'}\circ \dots \circ S_{j_{n}'}(s)=S_{j_1}\circ \dots \circ S_{j_m}(s).
	\end{equation}
	If $s=(0.t_1 t_2\dots)_q$, then from \eqref{eq:8-6-1} it follows that
	\[(0.j_1\dots j_m t_1t_2\dots)_q=S_{j_1}\circ \dots \circ S_{j_m}(s)=S_{j'_1}\circ \dots \circ S_{j'_{n}}(s)=(0.j'_1\dots j'_{n} t_1t_2\dots)_q,\]
	and thus $S_{j_{m+1}'}\circ\dots \circ S_{j_{n}'}(s)=s$, contradicting the minimality of $n$. Hence, $s\in L_n(s)$ which implies that there exist $i_1,\dots,i_n\in\{0,\dots,q-1\}$ such that $S_{i_1}\circ \dots \circ S_{i_n}(s)=s$. 
	Note that $(0.\overline{i_1\dots i_n})_q$ is also a fixed point of $S_{i_1}\circ\dots S_{i_n}$ but since $S_{i_1}\circ\dots S_{i_n}$ is a contraction on $[0,1]$ (with Lipschitz constant $q^{-n}$) it has a unique fixed point and we must have $s=(0.\overline{i_1\dots i_n})_q$.
	
	By definition $S_{i_n}(s)\in L_1(s)$ and from \eqref{eq:8-6-2} we deduce $J_s\geq J_{S_{i_n}(s)}$. Letting $x=S_{i_n}(s)$ in the left hand side of \eqref{eq:8-6-2} we have that $J_s\geq J_{S_{i_n}(s)}\geq J_{S_{i_{n-1}}\circ S_{i_n}(s)}$. Continuing in this way we see that 
	\[J_s\geq J_{S_{i_n}(s)}\geq \dots\geq J_{S_{i_2}\circ \dots \circ S_{i_n}(s)}\geq J_{S_{i_1}\circ \dots \circ S_{i_n}(s)}=J_{s}.\]
	This shows that the numbers $s,S_{i_n}(s),\dots, S_{i_2}\circ \dots \circ S_{i_n}(s)$ are discontinuities of $F$ with the same jump. By the minimality of $n$ these points are distinct and hence constitute a cycle.
	
\end{proof}	
Now, Theorem~\ref{thm1}(II) follows from Lemma~\ref{lem:b2} and the fact that $F$ has countably many points of discontinuity.

\subsubsection{Proof that $F_2$ satisfies \eqref{e:3}}
Because of Theorem~\ref{thm1}(II), in order to show that $F_2$ satisfies \eqref{e:3}, without loss of generality we may assume that
\begin{equation*}
F_2(x)=\frac{1}{n}\sum_{j=1}^n  H(x-s_j),
\end{equation*}
where $(s_1,\ldots,s_n)$ is a cycle.
For each $j \in \{1,\dots,n\}$, let $s_{j}(1)$ denote the first digit in the base-$q$ expansion of $s_j$ and note that $qs_j=s_{j-1}+s_{j}(1)$, where we define $s_0:=s_n$. Hence, for any $k\in \mathbb{Z}$, the characteristic function $f_2$ of $F_2$ satisfies 
\[ f_2(2\pi k q)=\frac{1}{n}\sum_{j=1}^n \mathrm{e}^{2\pi i k q s_j} =\frac{1}{n}\sum_{j=1}^n \mathrm{e}^{2\pi i k s_{j-1}}= f_2(2\pi k).\]
Then by Theorem~\ref{fourier}(I) it follows that $F_2$ satisfies \eqref{e:3}.


\bibliographystyle{APT}
\bibliography{bibliography}

\begin{thebibliography}{10}

\bibitem{BCM1997}
{\sc Barbaroux, J.-M., Combes, J.-M. and Montcho, R.} (1997).
\newblock Remarks on the relation between quantum dynamics and fractal spectra.
\newblock {\em J. Math. Anal. Appl.\/} {\bf 213,} 698--722.

\bibitem{Billingsley1965}
{\sc Billingsley, P.} (1965).
\newblock {\em Ergodic Theory and Information}.
\newblock Wiley.

\bibitem{Billingsley1995}
{\sc Billingsley, P.} (1995).
\newblock {\em Probability and Measure}.
\newblock Wiley Series in Probability and Statistics. Wiley.

\bibitem{Denjoy1932}
{\sc Denjoy, A.} (1932).
\newblock Sur quelques points de la th{\'e}orie des fonctions.
\newblock {\em C. R. Math. Acad. Sci. Paris\/} {\bf 194,} 44--46.

\bibitem{Denjoy1934}
{\sc Denjoy, A.} (1934).
\newblock Sur une fonction de {M}inkowski.
\newblock {\em C. R. Math. Acad. Sci. Paris\/} {\bf 198,} 44--47.

\bibitem{Dym1968}
{\sc Dym, H.} (1968).
\newblock On a class of monotone functions generated by ergodic sequences.
\newblock {\em Amer. Math. Monthly\/} {\bf 75,} 594--601.

\bibitem{Folland}
{\sc Folland, G.~B.} (1999).
\newblock {\em Real Analysis, Modern Techniques and Their Applications, 2nd
  ed.}
\newblock Wiley.

\bibitem{Harris1955}
{\sc Harris, T.~E.} (1955).
\newblock On chains of infinite order.
\newblock {\em Pacific J. Math.\/} {\bf 5,} 707--724.

\bibitem{Hu2001}
{\sc Hu, T.-Y.} (2001).
\newblock Asymptotic behavior of {F}ourier transforms of self-similar measures.
\newblock {\em Proc. Amer. Math. Soc.\/} {\bf 129,} 1713--1720.

\bibitem{Hutchinson1981}
{\sc Hutchinson, J.~E.} (1981).
\newblock Fractals and self similarity.
\newblock {\em Indiana Univ. Math. J.\/} {\bf 30,} 713--747.

\bibitem{JordanSahlsten}
{\sc Jordan, T. and Sahlsten, T.} (2016).
\newblock Fourier transforms of {G}ibbs measures for the {G}auss map.
\newblock {\em Math. Ann.\/} {\bf 364,} 983--1023.

\bibitem{Kairies1997}
{\sc Kairies, H.} (1997).
\newblock Functional equations for peculiar functions.
\newblock {\em Aequationes Math.\/} {\bf 53,} 207–241.

\bibitem{RLyons}
{\sc Lyons, R.} (1995).
\newblock Seventy years of {R}ajchman measures. {P}roceedings of the
  {C}onference in {H}onor of {J}ean-{P}ierre {K}ahane ({O}rsay, 1993).
\newblock {\em {J. Fourier Anal. Appl., Special Issue}\/} 363--377.

\bibitem{Minkowski1904}
{\sc Minkowski, H.} (1904).
\newblock {\em Zur Geometrie {d}er Zahlen, Verhandlungen {d}es III.
  Internationalen Mathematiker-Kongresses {i}n Heidelberg, 1904, {pp.}
  164--173. (Gesammelte Abhandlungen {v}on Hermann Minkowski. Bd. II. B. G.
  Teubner, Leipzig, 1911, {pp}. 43--52)}.
\newblock Reprinted by Chelsea, New York, 1967.

\bibitem{parry}
{\sc Parry, W.} (1960).
\newblock On the $\beta$-expansion of real numbers.
\newblock {\em Acta Math. Hungar.\/} {\bf 11,} 401--416.

\bibitem{Yuval}
{\sc Peres, Y., Schlag, W. and Solomyak, B.} (2000).
\newblock Sixty years of {B}ernoulli convolutions.
\newblock In {\em Fractal Geometry and Stochastics II}. ed. C.~Brandt, S.~Graf,
  and M.~Zähle.
\newblock vol.~46.
\newblock Birkhäuser Basel.
\newblock pp.~39--65.

\bibitem{Persson}
{\sc {Persson}, T.} (2015).
\newblock {On a problem by R. Salem concerning Minkowski's question mark
  function}.
\newblock {\em arXiv e-prints\/} arXiv:1501.00876.

\bibitem{Reese1989}
{\sc Reese, S.} (1989).
\newblock Some {F}ourier-{S}tieltjes coefficients revisited.
\newblock {\em Proc. Amer. Math. Soc.\/} {\bf 105,} 384--386.

\bibitem{Riesz1909}
{\sc Riesz, F.} (1909).
\newblock Sur les op\'{e}rations fonctionnelles lin\'{e}aires.
\newblock {\em C. R. Acad. Sci. Paris\/} {\bf 149,} 974--977.

\bibitem{Riesz1955}
{\sc Riesz, F. and Sz.-Nagy, B.} (1955).
\newblock {\em Functional Analysis}.
\newblock Dover Publications.

\bibitem{BabyRudin}
{\sc Rudin, W.} (1976).
\newblock {\em Principles of Mathematical Analysis}.
\newblock International Series in Pure and Applied Mathematics. McGraw-Hill.

\bibitem{Salem1943}
{\sc Salem, R.} (1943).
\newblock Some singular monotonic functions which are strictly increasing.
\newblock {\em Trans. Amer. Math. Soc.\/} {\bf 53,} 427--439.

\bibitem{Strichartz1990}
{\sc Strichartz, R.~S.} (1990).
\newblock Self-similar measures and their {F}ourier transforms, {I}.
\newblock {\em Indiana Univ. Math. J.\/} {\bf 39,} 797--817.

\bibitem{Takacs1978}
{\sc Tak{\'a}cs, L.} (1978).
\newblock An increasing continuous singular function.
\newblock {\em Amer. Math. Monthly\/} {\bf 85,} 35--37.

\bibitem{Varju}
{\sc Varjú, P.~P.} (2018).
\newblock Recent progress on {B}ernoulli convolutions.
\newblock In {\em Proceedings of the 7th European Congress of Mathematics}. ed.
  V.~Mehrmann and M.~Skutella.
\newblock American Mathematical Society Bookstore.
\newblock p.~847–867.

\end{thebibliography}

\end{document}